\documentclass[11pt]{amsart}

\usepackage{amscd,amssymb,amsopn,amsmath,amsthm,mathrsfs,graphics,amsfonts,enumerate,verbatim,calc,multicol}

\usepackage{bbm}
\usepackage{tikz}
\usepackage{lscape}
\usepackage{enumitem}
\usetikzlibrary{matrix,arrows,decorations.pathmorphing,cd}
\usepackage{hyperref}
\hypersetup{colorlinks=true,linkcolor={blue}}
\usepackage{cleveref}
\usepackage{url}

\usepackage{scalerel}
\usepackage{stackengine,wasysym}
\usetikzlibrary {positioning}
\usetikzlibrary{patterns}
\usetikzlibrary{calc}

\usepackage{subfigure}

\usepackage{comment} 

\usepackage{ dsfont }

\usepackage{color}

\usepackage[OT2,OT1]{fontenc}
\newcommand\cyr{%
\renewcommand\rmdefault{wncyr}%
\renewcommand\sfdefault{wncyss}%
\renewcommand\encodingdefault{OT2}%
\normalfont
\selectfont}
\DeclareTextFontCommand{\textcyr}{\cyr}

\usepackage{amssymb,amsmath}

\DeclareFontFamily{OT1}{rsfs}{}
\DeclareFontShape{OT1}{rsfs}{n}{it}{<-> rsfs10}{}
\DeclareMathAlphabet{\mathscr}{OT1}{rsfs}{n}{it}

\numberwithin{equation}{section}
\newtheorem{theorem}{Theorem}[section]
\newtheorem{lem}[theorem]{Lemma}
\newtheorem{cor}[theorem]{Corollary}

\newtheorem{quest}[theorem]{Question}
\newtheorem{prop}[theorem]{Proposition}

\theoremstyle{definition}

\theoremstyle{remark}
\newtheorem{remark}[theorem]{Remark}

\newtheorem{example}[theorem]{Example}

\newcommand{\ds}{\displaystyle}

\newcommand{\tr}{\operatorname{tr}}

\newcommand{\Ext}{\operatorname{Ext}}

\newcommand{\Hom}{\operatorname{Hom}}

\newcommand{\End}{\operatorname{End}}

\newcommand{\depth}{\operatorname{depth}}

\newcommand{\cC}{\mathcal{C}}

\newcommand{\m}{\mathfrak{m}}

\title{Taking trace does not preserve reflexivity}

\address{Harvey Mudd College}
\author[Lindo]{Haydee Lindo}
\email{hlindo@g.hmc.edu}

\author[Maitra]{Sarasij Maitra}
\address{University of Utah (Currently moved to Haverford College)}
\email{smaitra2@haverford.edu}

\author[Zhang]{William Zhang}
\address{University of Utah}
\email{u1397945@utah.edu}

\date{}

\subjclass[2020]{Primary 13C05, 13C13, 13B02, 13B22, 13F30, 13Gxx, 13H10}

\keywords{reflexive, trace ideal, partial trace ideal}

\begin{document}

\maketitle

\begin{abstract}
In this note, we address a question raised in \cite{DMS23} regarding the preservation of reflexivity under taking trace. We answer this question negatively. We also study a few cases where the question has a positive answer in a one dimensional, analytically unramified Cohen-Macaulay local ring.
\end{abstract} 

\section{Introduction}

Over a commutative ring $R$, a module $M$ is called reflexive if the natural map $M\to M^{**}$ is an isomorphism where $M^{**}$ denotes the double dual $\Hom_R(\Hom_R(M,R),R)$. Over a field, all finite dimensional vector spaces are reflexive modules. Similarly, over a general ring, all finite free modules are the simplest examples of reflexive modules. Also, it can be shown that any reflexive module occurs as a second syzygy module \cite[1.4.20]{BH93}, thereby lending some connection to the theory of resolutions as well.  The first formal treatment of reflexive modules seems to be present in \cite{bourbaki1965diviseurs}, though there were some discussions in the works of \cite{dieudonne1958remarks}, \cite{morita1958duality} and \cite{bass84} (where the term ``reflexive" was first used). In fact, over Gorenstein rings, a detailed analysis can be found in \cite{vasconcelos1968reflexive}. Using \cite[Proposition 1.4.1]{BH93}, the study of reflexivity mainly boils down to analyzing the case when $R$ is a one dimensional Cohen-Macaulay local ring and in recent years, various researchers have been actively exploring this avenue; see for instance \cite{kustin2021totally}, \cite{DMS23}, \cite{isobe2024reflexive}, \cite{DL:2024}, \cite{ENDO2024107662}, among other recent sources.

In such studies, the usage of trace ideal of a module is quite common and crucial. The trace ideal of a module is defined to be the sum of all homomorphic images of $M$ inside $R$. So in a way, it helps understand the properties of the module by looking at its \textit{total image} as an ideal in $R$. Many insightful and extremely useful results on trace ideals can be found in various sources such as \cite{lindo2017trace}, \cite{kobayashi2019rings}, \cite{goto2020correspondence}, \cite{lindo2022trace}, \cite{lyle2024annihilators} amongst others. In fact, various necessary and sufficient criteria for an ideal to be the trace ideal of a reflexive module were explored in \cite{lindo2017trace} paving the way for subsequent work.

In this article, our main aim was answering the following question raised in the recent works of Dao et. al.

\begin{quest}\cite[Question 7.17]{DMS23}\label{mainques} Let $R$ be a one dimensional Cohen-Macaulay local ring. If $I$ is a reflexive ideal, is $\tr_R(I)$ also reflexive?
\end{quest}

We settle this question negatively by providing a counterexample, namely we prove the following theorem (\Cref{main:counterexample}).

\begin{theorem}\label{introcounter}
    There exists an analytically unramified one dimensional local domain where trace of a reflexive ideal need not be reflexive.
\end{theorem}

In the course of this investigation we realized that \Cref{mainques} can have affirmative answers based on the co-length, $\ell(R/\cC)$, where $\cC$ is the conductor ideal of the integral closure $\overline{R}$ in $R$. Firstly, if this colength is small enough, then any regular trace ideal is reflexive (see \Cref{mainthm:smallcolengthrefimpliesrefftrace}).

\begin{theorem}\label{intromain1}
    Let $R$ be a one dimensional analytically unramified non-regular Cohen-Macaulay local ring with infinite residue field. Let $J$ be a proper regular trace ideal of $R$. Then the following statements hold. 
 
 \begin{enumerate}
     \item If $\ell(R/\cC)\leq 3$, then $J$ is reflexive.

     \item If $\ell(R/\cC)=4$ and $R$ has minimal multiplicity, then $J$ is reflexive.
 \end{enumerate}
\end{theorem}

Secondly, using the theory of partial trace ideals developed in \cite{Maitra:2022,Maitra:2024}, we were able to push the study further under some more constraints on $R$ (see \Cref{mainthm:trrefalways}).

\begin{theorem}\label{intromain2}
Let $R$ be an analytically unramified one dimensional non-regular local ring with infinite residue field such that $\overline{R}$ is a DVR . Let $I$ be a reflexive regular ideal of $R$. Then $\tr_R(I)$ is reflexive if any one of the following conditions holds.
    \begin{enumerate}
    \item $\ell(R/\cC)=4$, 
    \item $\ell(R/\cC)=5$ and $R$ is of minimal multiplicity.
    \end{enumerate} 
\end{theorem}

In fact, our counterexample to \Cref{mainques} occurs when we remove the minimal multiplicity assumption in \Cref{intromain2}, thereby making it a minimal counterexample in this context. 

The article is structured as follows. In \Cref{prelim}, we collect the various background results that we use throughout the paper. In \Cref{mainresults}, we prove both \Cref{intromain1} and \Cref{intromain2}. In the course of proving these, we recollect and prove various useful results concerning the relationship between trace, integral closure of an ideal and reflexivity (see \Cref{mainprop:intclosref}, \Cref{mainprop:intclosmax}). This naturally leads to \Cref{counterexample} where we prove \Cref{introcounter} thereby settling \Cref{mainques} negatively. Finally, in \Cref{misc}, we provide a collection of statements that provide further insights into double duals of trace ideals of finitely generated modules (see \Cref{misc:doubledualtrace}) and their connections with birational extensions in the form of endomorphism rings and their centers. This discussion is in the spirit of \cite{lindo2017trace}, \cite{goto2020correspondence} and \cite{ENDO2024107662}. 

\section*{Acknowledgements}
During the course of this project, S. Maitra was partially supported by Project No. 51006801 - American Mathematical Society-Simons Travel Grant. W. Zhang was supported by an REU grant from the University of Utah as well as by National Science Foundation Grant No. 1840190.

\section{Preliminaries}\label{prelim}
Throughout this article $(R,\m, k)$ will be a one dimensional commutative local Noetherian ring, $R$, with unique maximal ideal $\m$ and residue field $k$, which we assume to be infinite. A regular ideal will refer to an ideal $I$ containing a non-zero-divisor. The total ring of fractions of $R$ will be denoted by $Q$ and let $\overline{R}$ be the integral closure of $R$ in $Q$. All modules considered will be finitely generated (left) $R$-modules. For an $R$-module $M$, we use $\ell(M)$, respectively $\mu(M)$, to denote the length of the module, respectively minimal number of generators of the module.  By a fractional ideal, we mean an $R$-submodule of $Q$. If $I, J$ are fractional ideals of $R$, then \[I:J:=\{\alpha \in Q\mid \alpha J \subseteq I\}.\] We use $I:_R J$ to denote the case where we restrict $\alpha$ to be in $R$. This is called a colon ideal of $R$. 

A module $M$ is said to be maximal Cohen-Macaulay over $R$ if the depth of $M$ and the (Krull) dimension of $R$ are equal, where \[\depth(M):=\min\{i\mid \Ext^{i}_R(k,M)\neq 0\}.\] The ring $R$ is said to be Cohen-Macaulay if $R$ is a maximal Cohen-Macaulay module over itself. 
We refer the interested reader to \cite[Chap 1, 2]{BH93} for further details regarding the notions of depth and the Cohen-Macaulay property. We shall always assume that $\mu(\m)\geq 2$, that is, $R$ is not a regular local ring.

Let $\widehat{R}$ be the $\m$-adic completion of $R$. The ring $R$ is said to be analytically unramified if $\widehat{R}$ is a reduced ring, which implies $R$ is also reduced. In this case, $\overline{R}$ is a finitely generated module over $R$ \cite[Corollary 4.6.2]{HS06}.

Below we collect the necessary background tools that will be used in the subsequent discussion.

\subsection{Integral closure of Ideals and reductions}\label{prelimsec:intclos}
Given an ideal $I$, the integral closure of $I$, denoted $\overline{I}$, is defined to be the collection of all elements $r\in R$ such that $r$ satisfies a monic polynomial \[f(x)=\sum_{i=0}^n a_ix^{n-i},\] where $a_i\in I^i, i\geq 1$. The integral closure of $I$,  $\overline{I}$, is an ideal of $R$ and with $I\subseteq \overline{I}$. If $I=\overline{I}$, $I$ is said to be integrally closed. An element $x\in R$ (or the ideal $(x)$ generated by such an element)  is called a principal (alternatively, minimal) reduction of $I$ if $xI^n =I^{n+1}$ for some $n$. Under the assumption that the residue field $k$ is infinite,  every nonzero regular ideal has a minimal reduction generated by a regular element \cite[Proposition 8.3.7, Corollary 8.3.9]{HS06} (c.f. \cite[Sec 2]{DGH2001}). It is helpful to use minimal reductions as $\overline{xR}=\overline{I}$ (for a more general statement, see \cite[Corollary 1.2.5]{HS06}).

\subsection{Conductor Ideal}
The conductor ideal of $\overline{R}$ in $R$ is defined to be $\cC:=R:\overline{R}$. This is the largest ideal shared by both $R$ and $\overline{R}$, i.e., $\cC\overline{R}=\cC$ and any ideal $I$ such that $I\overline{R}=I$ necessarily means that $I\subseteq \cC$ (see, for instance \cite[Exercise 2.11]{HS06}). More generally, if $S$ is a \textit{birational extension} of $R$, i.e., a ring extension $R\subseteq S\subseteq Q$, then we define $\ds \cC_R(S):= R:S$ and call it the conductor of $S$ in $R$. Notice that $\cC_R(S)$ is regular if and only if $S$ is a finitely generated $R$-module: let $0\neq \alpha \in \cC_R(S)$. This implies that $  \alpha S \subseteq R\implies S\subseteq \frac{1}{\alpha}R$. Hence, $S$ must be a submodule of the finitely generated module $\frac{1}{\alpha} R$, thus must itself be finitely generated. Conversely, if $S$ is finitely generated as an $R$-module by the elements $\alpha_i\in Q$, then the least common multiple of the denominators of $\alpha_i$'s give rise to a regular element in $\cC_S(R)$. 

In particular, we will be interested in the case when $\overline{R}$ is finitely generated as an $R$-module (this is guaranteed, for instance when $R=\widehat{R}$ or $R$ is analytically unramified, as discussed above).

\subsection{Reflexive Modules and Ideals}
For any $R$-module $M$, there exists a natural $R$-linear map $\ds \Phi: M \longrightarrow \Hom_R(\Hom_R(M,R),R), 
~m \mapsto \phi_m$, where $\phi_m(f)=f(m)$ for any $f\in \Hom_R(M,R)$. The module $M$ is said to be \textit{reflexive} if $\Phi$ is an isomorphism of $R$-modules. It is standard practice in the literature to write $\Hom_R(M,R)$ as $M^*$ and we follow this convention in the rest of this paper.

We collect some important well-known observations in the following lemma which will be useful in the subsequent discussion. 

\begin{prop}\label{prelimprop:refprop}
Let $R$ be a one dimensional analytically unramified local Cohen-Macaulay ring. Let $M$ be a finitely generated $R$-module and let $I$ be fractional ideal of $R$ containing a regular element $x$ of $R$. Then the following statements hold.

\begin{enumerate}
    \item \cite[Proposition 2.2]{F19} The module $M^*$ is reflexive.

    \item \cite[Proposition 2.4(1)]{kobayashi2019rings}\cite[Lemma 2.4.3]{HS06} Let $J$ be a fractinal ideal of $R$. The module $\Hom_R(I,J)$ is isomorphic to $J:I$. Further, both the above modules are isomorphic to $\frac{1}{x}(xJ:_R I)$.
    
    \item \cite[Proposition 2.4(4)]{kobayashi2019rings} The fractional ideal $I$ is reflexive if and only if $\ds I=R:(R:I)$.

    \item \cite[Corollary 3.2]{DMS23} The ideals $\cC$ and $\m$ are reflexive.
    \end{enumerate}
\end{prop}

\begin{cor}\label{prelimcor:colonref}
    For any ideal $I$ with a regular element $x$, the ideal $xR:_R I$ is reflexive.
\end{cor}

\begin{proof}
Observe that $xR:_R I\cong \frac{1}{x}(xR:_R I)$ and the latter can be identified with $I^*$ using \Cref{prelimprop:refprop}(2). Now \Cref{prelimprop:refprop}(1) finishes the proof.
\end{proof}

In the context of \Cref{prelimprop:refprop}(2), any isomoprhism between two fractional ideal is given by multiplication by ana element of $Q$ \cite[Remark 2.2]{Maitra:2022}.

\subsection{Trace Ideal of a Module}
Given a module $M$, we define the trace ideal of $M$ as $$\ds \tr_R(M):=\sum_{f \in M^*}f(M).$$ An ideal $I$ is said to be a trace ideal if $I=\tr_R(M)$ for some module $M$. Notice that if $M\cong N$, then $\tr_R(M)=\tr_R(N)$, i.e., trace is invariant under isomorphisms. We record some important properties of trace ideals that we will use in this article. 

\begin{prop}\label{prelimprop:traceproperties}
    Let $R$ be a one dimensional analytically unramified local Cohen-Macaulay ring and let the conductor ideal be $\cC$. Let $M$ be a non-zero finitely generated $R$-module. The following statements hold.

    \begin{enumerate}
        \item \cite[Proposition 2.8 (iii)]{lindo2017trace} The equality $\tr_R(M)=R$ occurs if and only if $\ds M$ has $R$ a direct summand. In particular, if $M$ is an ideal, then $\tr_R(M)=R$ if and only if $M$ is a principal ideal.

        \item \cite[Proposition 2.8 (v)]{lindo2017trace}The inclusion $\tr_R(M)\subseteq \tr_R(M^*)$ holds and equality holds if and only if $M$ is reflexive.

        \item \cite[Proposition 2.4(2)]{kobayashi2019rings} If $M$ is a non-zero fractional ideal of $R$ containing a regular element of $R$, then $\tr_R(M)=(R:M)M$.

        \item \cite[Proposition 2.8 (iv)]{lindo2017trace} If $M$ is a regular ideal of $R$, then $M\subseteq \tr_R(M)$ with equality if and only if $M$ is a trace ideal.

         \item \cite[Corollary 3.6]{DMS23} If $M$ is an ideal with a principal reduction, then $\cC\subseteq \tr_R(M)$.

        \item \cite[Lemma 3.7]{DMS23} If $M$ is an ideal with  a regular element $x$, then $xR:_R M\subseteq \tr_R(M)$.

     \end{enumerate}
\end{prop}

\subsection{Partial Trace Ideals}
Given a module $M$ over a one dimensional analytically unramified local domain $R$, we let $\operatorname{h}(M):=\min\{\ell(R/I)\mid f(M)=I\text{~for some ~}f\in M^*\}.$ Any such ideal $I$ which achieves the value $\operatorname{h}(M)$ is called a partial trace ideal of $M$ (see \cite{Maitra:2022, Maitra:2024}). Notice that any such partial trace ideal of $M$ is of the form $f(M)$ for some $f\in M^*$ and is therefore contained within $\tr_R(M)$. Also, if $J$ is a partial trace ideal of $M$, then $J$ is a partial trace ideal of itself by definition \cite[Remark 2.3]{Maitra:2022}. Hence, it is often enough to restrict our attention to ideals.  Recall that if $R$ is a domain with field of fractions $Q$, then the rank of a module $M$ is defined to be $\dim_Q(Q\otimes_R M)$ \cite[Definition 1.4.2]{BH93}. Below are the main results that will be useful to us. 

\begin{prop}\label{prelimprop:partialtr}
    Let $R$ be an analytically unramified one dimensional local domain with integral closure $\overline{R}$.  

    \begin{enumerate}
    \item Let $J$ be any non-zero ideal of $R$.  Consider the following statements.
    \begin{multicols}{2}
    \begin{enumerate}
        \item $\operatorname{h}(J)=\ell(R/J)$,
        \columnbreak 
        \item $R: J\subseteq \overline{R}$.
    \end{enumerate}
    \end{multicols}
    Then $(b)$ implies $(a)$. Further if $\overline{R}$ is a $DVR$, then $(a)$ implies $(b)$.
    \item Let $J$ be a non-zero fractional ideal of $R$. Then for any partial trace ideal $I$ of $J$, $\ds \tr_R(I)=\tr_R(J)$. 

    \item If $\overline{R}$ is a $DVR$, then for any partial trace ideal $I$ of any module $M$, $\overline{\tr_R(I)} = \overline{I}$.
    \end{enumerate}

    \begin{proof}
        Statement (1) is directly stating \cite[Theorem 2.5]{Maitra:2022}. Statement (2) follows \cite[Proposition 3.5]{Maitra:2022} with the added observation that all non-zero fractional ideals have rank one. Statement (3) appears as part of the proof of \cite[Proposition 3.8]{Maitra:2024} and we repeat the proof here for convenience: by statement (1), we have $R:I\subseteq \overline{R}$. Thus, $(R:I)I\subseteq I\overline{R}\subseteq\overline{I\overline{R}}$. Thus $(R:I)I\subseteq \overline{I\overline{R}}\cap R$. Hence,  using \Cref{prelimprop:traceproperties}(3), (4) and \cite[Proposition 1.6.1]{HS06}, we obtain that $I\subseteq \tr_R(I)\subseteq \overline{I}$. Taking integral closure throughout now finishes the proof.
    \end{proof}
\end{prop}

\section{Main results}\label{mainresults}
Our primary goal in this section is to study some partial cases where we can guarantee that the trace ideal of a reflexive ideal is reflexive, in the situation when $(R,\m, k)$ is a one dimensional analytically unramified Cohen-Macaulay local ring. We first recall some known cases and also provide the proofs for convenience of the reader. 

\begin{lem}\label{mainlem:intclosred}
  Let $R$ be an analytically unramified one dimensional Cohen-Macaulay local ring with infinite residue field. Then for any regular ideal $I$, $\overline{I}=I\overline{R}\cap R$. In particular, any regular ideal shared by both $R$ and $\overline{R}$ is integrally closed.
\end{lem}

\begin{proof}
    By \cite[Proposition 1.6.1]{HS06}, it is enough to show that $\overline{I\overline{R}}=I\overline{R}$. Since the residue field is infinite, there exists a minimal reduction $x$ of $I$. Since $x\in I$, we get that $x\overline{R}\cap R\subseteq I\overline{R}\cap R\subseteq \overline{I\overline{R}}\cap R=\overline{I}$. However, since $\overline{R}$ is integrally closed, we get that $x\overline{R}=\overline{x\overline{R}}$ \cite[Proposition 1.5.2]{HS06}. Now the fact that $\overline{xR}=\overline{I}$, finishes the first part of the proof. For the last statement, we use the fact that $I\overline{R}=I$ whenever $I$ is shared by both $R$ and $\overline{R}$.
\end{proof}

 \begin{prop}\label{mainprop:intclosref}
Let $R$ be an analytically unramified one dimensional Cohen-Macaulay ring with conductor ideal $\cC$ and infinite residue field. Then any integrally closed ideal containing $\cC$ is a reflexive trace ideal. 
\end{prop}

\begin{proof}
    Since $R$ is analytically unramified, we know that $\cC$ is a proper regular ideal. The proof now follows from \Cref{mainlem:intclosred} and \cite[Proposition 3.9]{DMS23} by taking $S=\overline{R}$ and $I$ to be any integrally closed ideal containing $\cC$. 
\end{proof}

\begin{cor}\label{maincor:intclosref}
    Let $R$ be an analytically unramified one dimensional Cohen-Macaulay ring with infinite residue field. Let $J$ be a regular trace ideal of $R$. Then the integral closure $\overline{J}$ is a reflexive trace ideal and $\tr_R(\overline{J}^*)=\overline{J}$.
\end{cor}

\begin{proof}
    Notice that $J$ contains the conductor ideal by \Cref{prelimprop:traceproperties}(5). The proof now follows immediately from \Cref{mainprop:intclosref} and  \Cref{prelimprop:traceproperties} $(2),(4)$.
\end{proof}

We should also mention that \cite[Theorem 6.2]{DL:2024} provides a strong result that the integral closure of a trace ideal is again a trace ideal.

Recall that a Cohen-Macaulay local ring $R$ of dimension $\dim(R)$ is of \textit{minimal multiplicity} if $e(R)=\mu(\m)+\dim(R)-1$ where $e(R)$ denotes the Hilbert Samuel multiplicity of $R$ (for further details, we refer the reader to \cite[Chapter 11]{HS06}, \cite[1]{abhyankar1967local}). We shall use the following characterization of minimal multiplicity: if the residue field $k$ is infinite, then $R$ is of minimal multiplicity if and only if there exists a minimal reduction $x$ of $\m$ such that $\m^2=x\m$ \cite[Theorem 1]{sally1977associated}.  The following lemma is well known in the literature. We provide a proof here. Recall that we are always assuming that $R$ is non-regular, i.e., $\mu(\m)\geq 2$.
\begin{lem}\label{mainlem:minmult}
    Let $(R,\m, k)$ be a one dimensional non-regular Cohen-Macaulay ring with infinite residue field. Then $R$ is of minimal multiplicity if and only if $xR:_R \m=\m$ for some minimal reduction $x$ of $\m$. 
\end{lem}

\begin{proof}
    Observe that $\tr_R(\m)=\m$ as otherwise $\tr_R(\m)=R$ and \Cref{prelimprop:traceproperties}(1) shows that $\m$ must be principal, a contradiction to the regularity assumption. Now observe that $R$ is of minimal multiplicity if and only if $\m^2=x\m$ for some minimal reduction $x$ of $\m$ \cite[Theorem 1]{sally1977associated}. If $\m^2=x\m$, then \cite[Corollary 3.8]{DMS23} shows that $\m=xR:_R \m$. Conversely, assume that $xR:_R \m=\m$. This implies that $\m^2 \subseteq xR$. Now \cite[Proposition 8.3.3]{HS06} shows that $\m^2 = xR\cap\m^2 = x\m$ since $xR$ is a minimal reduction of $\m$. This finishes the proof.
\end{proof}

Recall that colons satisfy the following nice inclusion reversing property: if $J\subseteq K$, then $I:_R K\subseteq I:_R J$. We use this to establish the next result. 

\begin{prop}\label{mainprop:intclosmax}
    Let $R$ be an analytically unramified one dimensional non-regular Cohen-Macaulay ring of minimal multiplicity and infinite residue field. If $I$ is a proper regular trace ideal such that $\overline{I}=\m$, then $I=\m$ and hence reflexive. 
\end{prop}

\begin{proof}
    Since $\overline{I}=\m$ and $I\subseteq \m$, we can find a common minimal reduction $xR$ of $I$ and $\m$. Since $R$ is of minimal multiplicity, by \Cref{mainlem:minmult}, $xR:_R \m=\m$. Since $I\subseteq \overline{I}$, the conclusion now follows from \Cref{prelimprop:refprop}(4) and the following chain of ideals: $$\m=xR:_R\m=xR:_R\overline{I}\subseteq xR:_R I\subseteq I\subseteq \m.$$
\end{proof}

\begin{remark}
    Observe that we used the case, $\m^2=x\m$ in the above discussion and also the fact that $\tr_R(\m)=\m$. More generally, under the same hypothesis on the ring as above, if $I$ is a regular stable ideal, i.e., $I\cong \Hom_R(I,I)$, then we know that in this case, $I^2=xI$ for some non-zero divisor $x$, and hence $\tr_R(I)=x:_RI\cong I^*$ \cite[Proposition 3.2, Proposition 3.10]{DL:2024}. Reflexivity now follows from \Cref{prelimprop:refprop}(1).
\end{remark} 

We now provide a series of statements based on the co-length of the conductor ideal, which enables us to conclude reflexivity under taking trace in some cases.



\begin{theorem}\label{mainthm:trrefalways}Let $R$ be a one dimensional analytically unramified non-regular Cohen-Macaulay local ring with infinite residue field. Let $J$ be a proper regular trace ideal of $R$. Then the following statements hold. 
 
 \begin{enumerate}
     \item If $\ell(R/\cC)\leq 3$, then $J$ is reflexive.

     \item If $\ell(R/\cC)=4$ and $R$ has minimal multiplicity, then $J$ is reflexive.
 \end{enumerate}
    \end{theorem}
    \begin{proof}
        Since $R$ is analytically unramified, $\overline{R}$ is finitely generated over $R$ and hence $\cC=R:\overline{R}$ is a non-zero regular ideal of $R$. Since the residue field is infinite, we know that every proper regular ideal has a principal reduction (see \Cref{prelimsec:intclos}). By \Cref{prelimprop:traceproperties}(5), $\cC\subseteq J$ for any regular trace ideal $J$ of $R$. 

        If $\ell(R/\cC)\leq 2$, then 
the only ideals containing $\cC$ are $\cC$, $\m$ and $R$. Hence, any proper regular trace ideal is reflexive by \Cref{prelimprop:refprop}(4). 

Now assume $\ell(R/\cC)=3$. Since the residue field is infinite, we can take $xR$ to be a minimal reduction of $J$ and $\overline{J}$. By \cite[Theorem 4.6(1), Corollary 4.10]{DMS23}, we know that $J\cC=x\cC\subseteq xR$ and $\overline{J}\cC=x\cC$. Hence using \Cref{prelimprop:traceproperties}(6), we get that $\cC\subseteq xR:_R J\subseteq \tr_R(J)=J\subseteq \m\subseteq R$. Since $\ell(R/\cC)=3$, there are three possibilities: $xR:_R J=\cC$ or $xR:_R J=J$ or $J=\m$. In the latter two cases, we get the conclusion from \Cref{prelimcor:colonref} and \Cref{prelimprop:refprop}(4). Thus, assume that $xR:_R J=\cC$. 

Since we also have $\overline{J}\cC = x\cC$, we obtain that $\cC\subseteq xR:_R \overline{J}\subseteq xR:_R J.$ Thus, we get that $\cC=xR:_R \overline{J}\cong \overline{J}^*$. Taking trace on both sides and using \Cref{maincor:intclosref}, we obtain that $\cC=\overline{J}$. Hence $\cC=J=\overline{J}$  and thus $J$ is reflexive, finishing the proof of $(1)$. 

For $(2)$, observe again that we have the chain $\cC\subseteq xR:_R J\subseteq J\subseteq \overline{J}\subseteq \m$. If $\cC=xR:_R J$, then the same argument as above shows that $J=\cC$ and hence reflexive. We can also assume that $xR:_R J\subsetneq J \subsetneq \overline{J}$ as otherwise our conclusion holds from \Cref{prelimcor:colonref} and \Cref{maincor:intclosref}. Thus, we have the chain $\cC\subsetneq xR:_R J\subsetneq J\subsetneq \overline{J}\subseteq \m$. So the only possibility that remains is $\overline{J}=\m$ and we are done in this case by \Cref{mainprop:intclosmax}.
\end{proof}

The following example shows that the minimal multiplicity assumption in \Cref{mainthm:trrefalways}(2) is essential.

\begin{example}\cite[Example 7.12]{DMS23}
    Let $R=k[[t^{5},t^6,t^7]]$. Let $I=(t^5,t^7)$. Then $I$ is a non-reflexive trace ideal.
\end{example}
We now shift our focus on starting with a reflexive ideal and studying its trace. We can constrain our hypothesis a bit more to increase our colength study further. The following lemma will be important in the subsequent theorem. 

\begin{lem}\label{mainlem:chain}
    Let $R$ be an analytically unramified one dimensional non-regular local ring with infinite residue field such that $\overline{R}$ is a DVR. For any finitely generated non-zero fractional ideal $M$, let $I$ be a partial trace ideal of $M$ and let $xR$ be a minimal reduction of $I$. Then the following chain of ideals exists:
    \begin{equation}\label{eq:chain}
         \cC\subseteq xR:_R \overline{I}=xR:_R \overline{\tr_R(M)}\subseteq  xR:_R \tr_R(M)\subseteq xR:_R I \subseteq \tr_R(M) \subseteq \overline{\tr_R(M)} =\overline{I}\subseteq \m. 
    \end{equation}
\end{lem}

\begin{proof}
    Since $\overline{R}$ is a $DVR$, $R$ is a Cohen-Macaulay local domain. Let $L=\tr_R(M)$. By \Cref{prelimprop:partialtr}(2) and (3), we know that $\tr_R(I)=L\subseteq \overline{L}=\overline{I}\subseteq \m$. Since $I\subseteq \tr_R(I)$, we can choose $xR$ to be a common minimal reduction of $I, L$ and $\overline{I}$. The same argument as in the proof of \Cref{mainthm:trrefalways} gives us that $\cC\subseteq xR:_R \overline{I}=xR:_R\overline{L}$. Since $I\subseteq L\subseteq \overline{L}$, we get $xR:_R \overline{L}\subseteq xR:_R L\subseteq xR:_R I\subseteq \tr_R(I)=L$ where the last inclusion follows from \Cref{prelimprop:traceproperties}(6). Now connecting these chains finishes the proof. 
\end{proof}
\begin{theorem}\label{mainthm:smallcolengthrefimpliesrefftrace}
    Let $R$ be an analytically unramified one dimensional non-regular local ring with infinite residue field such that $\overline{R}$ is a DVR . Let $I$ be a reflexive regular ideal of $R$. Then $\tr_R(I)$ is reflexive if any one of the following conditions holds.
    \begin{enumerate}
    \item $\ell(R/\cC)=4$, 
    \item $\ell(R/\cC)=5$ and $R$ is of minimal multiplicity.
    \end{enumerate} 
\end{theorem}

\begin{proof}
    Let $L=\tr_R(I)$. Since $\tr_R(I)$ and reflexivity are invariant under the isomorphism class of $I$, we can replace $I$ to be a partial trace ideal of itself. By \Cref{prelimprop:partialtr}(1), we get that $R:I\subseteq \overline{R}$ and since $I$ is reflexive, using \Cref{prelimprop:refprop}(3), we get that $\cC=R:\overline{R}\subseteq R:(R:I)=I$. Hence, by \Cref{mainprop:intclosref}, $\overline{I}$ is a reflexive trace ideal. Moreover, from \Cref{prelimprop:partialtr}(3), we get that $L=\tr_R(I)\subseteq \overline{L}=\overline{I}$.
    
For $(1)$, start with the following chain from \Cref{eq:chain}: $\cC\subseteq xR:_R \overline{I}\subseteq xR:_R I\subseteq L\subseteq \m$. Since $\ell(\m/\cC)=3$, we get that at least one of the inclusions is an equality. If $\cC=xR:_R \overline{I}$, then $\cC\cong \overline{I}^*$. Taking trace, we conclude that $\cC=\overline{I}$ using \Cref{mainprop:intclosref} and \Cref{prelimprop:traceproperties}(2). But this shows that $\cC=I=L$ and hence $L$ is reflexive. Next, let $xR:_R \overline{I}=xR:_R I$. By \cite[Remark 2.4]{DMS23}, we get that $x\overline{I}^*=xI^*$, but this means that $\overline{I}=I$ since $I$ and $\overline{I}$ are both reflexive. But now from $I\subseteq L\subseteq \overline{I}$, we conclude that $L=I=\overline{I}$ and hence $L$ is reflexive. If $L=\m$, we are done by \Cref{prelimprop:refprop}(4).
Finally, if $xR:_R {I}=L$, then $L\cong {I}^*$ and hence is reflexive by \Cref{prelimprop:refprop}(1). 
This finishes the proof of $(1)$. 

For $(2)$, from \Cref{eq:chain} we now start with \begin{equation}\label{eqref}\cC\subseteq xR:_R \overline{I}\subseteq xR:_ R I\subseteq L\subseteq \overline{L}\subseteq \m.\end{equation} 
Since $\ell(\m/\cC)=4$, we need at least one inclusion to be an equality.  From the arguments in the previous paragraph, we may assume that $L=\overline{L}$ or $\overline{L}=\m$. In the first case, \Cref{maincor:intclosref} settles the problem. In the latter case, we obtain that $L=\m$ by \Cref{mainprop:intclosmax}. This completes the proof.
\end{proof}

\section{A Counterexample to \Cref{mainques}}\label{counterexample}

In this section, we provide an example to show the necessity of minimal multiplicity in \Cref{mainthm:smallcolengthrefimpliesrefftrace}. More importantly, this provides a counterexample to \Cref{mainques}. Our counterexample arises in the context of studying complete  numerical semigroup rings which are one dimensional local domains of the form $k[[t^{a_1},t^{a_2},\ldots, t^{a_n}]]$ where $k$ is a field and $ a_1<\ldots<a_n, a_i\in \mathbb{N}.$ In this case, $\overline{R}=k[[t]]$ which is a $DVR$ and the conductor ideal $\cC$ can be described as the collection of all $t^{c+i}, i\geq 0$ where $c$ is such that $t^{c-1}\notin R$ but $t^{c+i}\in R$ for all $i\geq 0$. The value group of $R$, denoted $v(R)$ is the collection of all integers $r$ such that $t^r\in R$. More generally, for a fractional ideal $I$, the value group of $I$ is defined to be $v(I):=\{r\in \mathbb{Z}\mid t^r\in I\}$. Notice that $v(\m)=v(R)$. 

A minimal reduction of the maximal ideal $(t^{a_1},\ldots, t^{a_n})$ is given by $t^{a_1}$ \cite[Remark 2.1]{maitra2023extremal}. The advantage of working on these rings is that a lot of information about ideals is encoded in the associated  valuation semigroup of the ring. Thus, often the study boils down to studying the semigroup generated by the positive integers $a_1,\ldots, a_n$ . We refer the reader to the discussion in \cite[Sec 3.3]{lyle2024annihilators}; most importantly, every homogeneous element is essentially of the form $t^h$ and hence to detect the presence of such an element in a desired subset of elements of $R$, it is enough to detect the presence of the integer $h$ in the corresponding suitable semigroup of integers associated with the subset.

\begin{lem}\label{lem:scalarandcolonproperty}
    Let $R$ be an integral domain with fraction field $Q$. For any non-zero $\alpha\in Q$ and any fractional ideal $L$, $\ds \alpha(R: L)=R:\frac{1}{\alpha}L.$
\end{lem}
\begin{proof}
    Notice that $x\in \alpha(R:L)\iff \frac{x}{\alpha} L\subseteq R \iff \frac{x}{\alpha}\in R:L \iff x\in \alpha(R:L)$.
\end{proof}
\begin{theorem}\label{main:counterexample}
    Let $\ds R=k[[t^7, t^8, t^9, t^{11}]]$ where $k$ is infinite. Let $x_1=t^7,x_2=t^8,x_3=t^9$ and $x_4=t^{11}$. Let $\cC$ be the conductor ideal. Then the following statements hold.

    \begin{enumerate}
    \item $R$ is not of minimal multiplicity and $\ell(R/\cC)=5$.
        \item The ideal $I=(x_2,x_3,x_1^3)$ is reflexive.
        \item $\tr_R(I)=(x_1,x_2,x_3)$.
        \item $\tr_R(I)$ is not a reflexive ideal.
    \end{enumerate}
    In particular, \Cref{mainques} is false in general.
\end{theorem}

\begin{proof}
    Notice that $x_1R:_R \m$ is given by all elements of the form $t^h\in R$ such that $t^h\m\subseteq x_1R$. In terms of valuations, this is the same as studying whether $h+r-7\in v(\m)$ for all $r\in v(\m)$. However, observe that $x_2\in x_1R:_R\m$ if and only if $8+r-7\in v(\m)$ for all $r$, which is a contradiction since $r=9\in v(\m)$ but $10\notin v(\m)$. Thus, $x_2\notin x_1R:_R \m$ which shows that $x_1R:_R \m \subsetneq \m$ and hence $R$ is not of minimal multiplicity by \Cref{mainlem:minmult}. 

    Now notice that $t^{13}\notin R$ but $t^{14+i}\in R$ for all $i\geq 0$. Hence $v(\cC)=\{14, 15, \longrightarrow\}$. Moreover, $v(R)=\{0,7,8,9,11,14,15,\longrightarrow\}$. Thus, $\ell(R/\cC)=5$, being given by the valuations $\{0,7,8,9,11\}$ \cite[Proposition 2.2]{maitra2023extremal} (c.f. \cite[Proposition 2.9]{HK1971}). This finishes the proof of $(1)$.

        For (2), we begin by showing that $I$ can be realized as the double dual $J^{**}$ where $J=(x_2,x_3)=(t^8,t^9)$.  Since $J=x_2R+x_3R$, we get that $J^*=R:J=(R:x_2R)\cap (R:x_3R)=\left(\frac{1}{x_2}(R:R)\right)\cap \left(\frac{1}{x_3}(R:R)\right)$ where the last equality follows from \Cref{lem:scalarandcolonproperty}. Hence, $v(J^*)$ is given by $v(t^{-8}R)\cap v(t^{-9}R)$. Since we know $v(R)$, we get that $v(t^{-8}R)=\{r-8\mid r\in v(R)\}$ and $v(t^{-9}R)=\{r-9\mid r\in v(R)\}$. The intersection is easily seen to be $\{-1,0,6,7,8,\longrightarrow\}$. Notice that using $\{t^{-1},t^{0},t^{12}\}$ as $R$-module generators, all other elements of the form $t^j$, with $j$ in the above intersection, can be generated. Hence, $J^*=t^{-1}R+R+t^{12}R$.

        
        Next, we compute $J^{**}=R:J^{*}=(R:t^{-1}R)\cap (R:R)\cap (R:t^{12}R)=(t(R:R))\cap R\cap (t^{-12}(R:R))=tR\cap R\cap t^{-12}R$, where the second to last inequality again follows from \Cref{lem:scalarandcolonproperty}. This is the same as looking for the valuations in $v(tR)\cap v(R)\cap v(t^{-12}R)$. Now observe that $v(tR)=\{1,8,9,10,12,15,\longrightarrow\}$ and $v(t^{-12}R)=\{-12,-5,-4,-3,-1,2,
        \longrightarrow\}$. Thus,   $v(J^{**})=v(tR)\cap v(R)\cap v(t^{12}R)=\{8,9,15,\longrightarrow\}$. Finally, observe that $\{t^8,t^9,t^{21}\}$ as $R$-module generators, generates all the remaining elements whose valuations are in the intersection. Thus, $J^{**}=(t^8,t^9,t^{21})=(x_2,x_3,x_1^3)$ as desired. This shows that $I$ is a reflexive ideal, finishing the proof of (2).

    To see (3), we first observe that $(x_2,x_3)\subseteq \tr_R(I)$ by \Cref{prelimprop:traceproperties}(4). Further notice that $x_1x_3=t^{16}=x_2^2 \in x_2R$ and $x_1x_1^3=t^{28}=x_2x_3x_4\in x_2R$. Since $x_2R:_R I=(x_2R:_R x_2R)\cap (x_2R:_Rx_3R)\cap(x_2R:_Rx_1^3R)=(x_2R:_Rx_3R)\cap(x_2R:_Rx_1^3R)$, we conclude that $x_1\in x_2R:_R I \subseteq \tr_R(I)$ where the last inclusion is from \Cref{prelimprop:traceproperties}(6). Thus, we get $(x_1,x_2,x_3)\subseteq \tr_R(I)$. Next notice that $\{t^{14},t^{15},\longrightarrow\}\subseteq (x_1,x_2,x_3)$. So, $(x_1,x_2,x_3) = \tr_R(I)$ if and only if $x_4\notin \tr_R(I)$.


    From \Cref{prelimprop:traceproperties}(3), we know that $\tr_R(I)=I^*I$. We first compute $I^*$ using the procedure as in the proof of (2): $I^*=(R:x_2R)\cap (R:x_3R)\cap (R:x_1^3R)=t^{-8}R\cap t^{-9}R\cap t^{-21}R$. Since $v(t^{-21}R)=\{-21,-14,-13,-12,-10,-7,\longrightarrow\}$, we find that the intersection of the valuations is given by $\{-1,0,6,7,\longrightarrow\}$. Hence $I^*=t^{-1}R+R+t^{12}R$ using the same argument as before. So, $\tr_R(I)=I^*I=(t^{-1}R+R+t^{12}R)(t^{8}R+t^{9}R+t^{21}R)$ and none of the combinations give rise to the valuation $11$ since $t^{12}\notin R$. Thus $x_4\notin \tr_R(I)$ proving (3).


    
    Finally, we claim that $x_4\in \tr_R(I)^{**}$.  
Let $L=\tr_R(I)=(x_1,x_2,x_3)$. Then $L^{*}=x_1^{-1}R\cap x_2^{-1}R\cap x_3^{-1}R$. Again, observing the valuations, we obtain that $L^{*}=R+\sum_{i\geq 7}t^iR$. Looking at the generators, we conclude that $x_4L=t^{11}L\subseteq R$, hence $x_4\in L^{**}$. But $x_4\not \in L$. Thus, $\tr_R(I)$ is not reflexive by \Cref{prelimprop:refprop}(3), and this finishes the proof of (4).
\end{proof}

The previous example shows that without the minimal multiplicity assumption, \Cref{mainques} is false. However, our tools do not seem to answer the case when $R$ is of minimal multiplicity. As such, we end this section with the following question.

\begin{quest}
    Let $R$ be a one dimensional Cohen Macaulay ring of minimal multiplicity. If $I\subseteq R$ is a reflexive ideal, is $\tr_R(I)$ reflexive? 
\end{quest}

\section{Miscellaneous Results}\label{misc}
In this section, we provide some miscellaneous results that provide  further tools to study the relationship between a trace ideal and its double dual. We start by providing a more general version of \Cref{prelimprop:refprop}(1) by relaxing the hypothesis. Recall that a Noetherian ring $R$ satisfies Serre's criterion $(S_k)$ if for all prime ideals $P$ in $R$, the depth of $R_P$ is at least $\min\{k,\operatorname{height} P\}$ (see \cite[Sec 4.5]{HS06}) for further details). Also, recall that a module $M$ is called totally reflexive if $M$ is reflexive and $\Ext^i_R(M,R)=\Ext^i_R(M^*,R)=0$ for all $i>0$ (refer to \cite{kustin2021totally} for instance). 

\begin{prop}\label{lem:dualreflexive}
    Let $R$ be a reduced local Cohen-Macaulay ring or a ring satisfying $(S_1)$. Let $M$ be a finitely generated module over $R$. If $R$ only satisfies $(S_1)$ then further assume that $M$ and $M^*$ are locally totally reflexive on all the minimal primes of $R$. Then $\displaystyle \tr_R(M^*)=\tr_R(M^{**})$.
\end{prop}

\begin{proof}
    Under the given hypotheses, $M^*$ is reflexive (see \cite[Lemma 2.5]{DMS23} and \cite[Proposition 2.2]{F19}). Now the conclusion follows from \cite[Proposition 2.8]{lindo2017trace}.
\end{proof}

\begin{prop}\label{prop:birexttrace}
    Let $S$ be a finite birational extension of $R$ where $R$ is a reduced local Cohen-Macaulay ring. Then 
$\tr_R(S)=\tr_R(S^*)=\tr_R(S^{**})=S^{*}$.
\end{prop}

\begin{proof}
    Let $\cC_R(S)=R:S=S^*$ denote the conductor ideal of $S$ in $R$. Since it is the largest common ideal between $R$ and $S$, notice that $\tr_R(S)=S^*S=\cC_R(S)=S^*$. This shows that $\tr_R(S)=S^*$. Now applying $\tr_R(\cdot)$ again to this equality and using \Cref{lem:dualreflexive}, we get that $\tr_R(S)=\tr_R(\tr_R(S))=\tr_R(S^*)=\tr_R(S^{**})$ which finishes the proof.
\end{proof}

\begin{cor}\label{cor:dualtrace}
Let $M$ be a finitely generated module over a reduced local Cohen-Macaulay ring such that $\tr_R(M)$ is a regular ideal. Then $\tr_R(M)^{**}=\tr_R(\End_R(\tr_R(M)))$. 
\end{cor}

\begin{proof}
    Let $J=\tr_R(M)$ and consider the birational extension $S=\End_R(J)$. Then $J^*=R:J=J:J=S$ since $J$ is a trace ideal \cite[Proposition 2.4]{kobayashi2019rings}. Thus $J^{**}=S^*=\tr_R(S)$ by \Cref{prop:birexttrace}, thereby finishing the proof.
\end{proof}

\begin{cor}\label{misc:doubledualtrace}
     Let $R$ be a reduced local Cohen-Macaulay ring. Then the reflexive hull of a regular trace ideal is again a trace ideal. 
\end{cor}

\begin{proof}
    The proof is immediate from \Cref{cor:dualtrace}.
\end{proof}

In the following, we denote the center of the endomorphism ring of a module $M$, by $Z(\End_R(M))$.

\begin{prop}\label{appprop:traceandcenter}
    Let $R$ be a reduced Cohen-Macaulay local ring and let $M$ be a finitely generated reflexive $R$-module with regular trace ideal. Then $\tr_R(M)$ is reflexive if and only if $\tr_R(M)=\tr_R(Z(\End_R(M)))$. 
\end{prop}

\begin{proof}
    By \cite[Corollary 3.11]{lindo2017trace}, we know that $\tr_R(\End_R(\tr_R(M)))=\tr_R(Z(\End_R(M)))$. The result now follows immediately from \Cref{cor:dualtrace}.
\end{proof}

\begin{cor}
    Let $R$ be a reduced local Cohen-Macaulay ring. If $I$ is a reflexive ideal of $R$, then $\tr_R(I)$ is reflexive if and only if $\tr_R(I)=\tr_R(\End_R(I)).$
\end{cor}

\begin{proof}
    Since $\End_R(I)=I:I$ is a commutative ring, we get that $Z(\End_R(I))=\End_R(I)$. The proof now follows immediately from \Cref{appprop:traceandcenter}. 
\end{proof}

\bibliography{references}
\bibliographystyle{alpha}

\end{document}